\documentclass[12pt]{amsart}
\usepackage{amsmath,amssymb,mathrsfs,hyperref,color}

\usepackage{subfigure}
\usepackage{float}
\usepackage{times}
\usepackage{tikz}
\usepackage{cases}
\usetikzlibrary{intersections}
\usepackage{paralist}
\usepackage{xcolor}

\usetikzlibrary{shadings}
\usetikzlibrary[intersections]
\usetikzlibrary{arrows,shapes}

\makeatletter
\def\setliststart#1{\setcounter{\@listctr}{#1}%
  \addtocounter{\@listctr}{-1}}
\makeatother

\makeatletter
\@addtoreset{figure}{section}
\makeatother

\usepackage{calc}
\newtheorem{The}{Theorem}[section]

\newtheorem{Lem}[The]{Lemma}
\newtheorem{Pro}[The]{Proposition}
\theoremstyle{definition}

\theoremstyle{remark}
\newtheorem{Rem}[The]{Remark}

\numberwithin{equation}{section}

\newcommand{\T}{\mathbb{T}}
\newcommand{\R}{\mathbb{R}}

\newcommand{\N}{\mathbb{N}}

\newcommand{\SING}{\mbox{\rm Sing}\,(u)}
\newcommand{\CUT}{\mbox{\rm Cut}\,(u)}
\newcommand{\BSING}{\overline{\SING}}
\newcommand{\IU}{\mathcal{I}(u)}

\title[On and beyond propagation of singularities]{On and beyond propagation of singularities of viscosity solutions}
\author{Piermarco Cannarsa \and Wei Cheng}
\address{Dipartimento di Matematica, Universit\`a di Roma ``Tor Vergata'',
Via della Ricerca Scientifica 1, 00133 Roma, Italy}
\email{cannarsa@mat.uniroma2.it}
\address{Department of Mathematics, Nanjing University, Nanjing 210093, China}
\email{chengwei@nju.edu.cn}
\date{\today}
\begin{document}

\begin{abstract}
	This is a survey paper for the recent results on and beyond propagation of singularities of viscosity solutions. We also collect some open problems in this topic.
\end{abstract}

\maketitle

\section{Introduction}

It is commonly accepted that, Hamilton-Jacobi equations in the form
\begin{equation}\label{eq:static_intro}
	H(x,Du(x))=c,\quad (x\in M)
\end{equation}
or
\begin{equation}\label{eq:evolutionary_intro}
	D_tu(t,x)+H(x,D_xu(t,x))=0\quad(x\in M)
\end{equation}
plays an important role in many fields such as PDEs, calculus of variation and optimal control, Hamiltonian dynamical systems and Riemannian geometry. Here $M$ is a smooth manifold without boundary and $c\in\R$ is the so called Ma\~n\'e critical value. The notion of viscosity solutions of Hamilton-Jacobi equations, introduced in the seminal papers \cite{Crandall_Lions1983} and \cite{Crandall_Evans_Lions1984}, provides the right class of generalized solutions to study existence, uniqueness, and stability issues  under wide classes of boundary-initial conditions.

The study of the propagation of singularities of viscosity solutions is a kind of finer  analysis of the associated problem of calculus of variations and optimal controls. We suppose that $H$ in \eqref{eq:static_intro} is a Tonelli-like Hamiltonian. Since any viscosity solution $u$ of \eqref{eq:static_intro} is locally semiconcave (with linear modulus), we denote by $D^+u(x)$ the superdifferential of $u$ at $x$. We say $x\in M$ a singular point of $u$ if $D^+u(x)$ is not a singleton. The set of all singular points of $u$ is denoted by $\SING$.

%Although significant contributions investigating the singularities of viscosity solutions were already given in \cite{Cannarsa-Soner} and \cite{Ambrosio-Cannarsa-Soner}, the current approach to this problem goes back to \cite{alca99}, where the propagation of singularities was studied for general semiconcave functions.

A specific approach to the problem on the propagation of singularities was developed in \cite{Albano_Cannarsa2002} by solving the generalized characteristic inclusion
\begin{equation}\label{eq:gc_intro}
	\dot{\mathbf{x}}(s)\in\mathrm{co}\, H_p\big(\mathbf{x}(s),D^+u(\mathbf{x}(s))\big),\quad \text{a.e.}\;s\in[0,\tau]\,.
\end{equation}
More precisely, if the initial point $x_0$ belongs to $\SING$, and is not a critical point of $u$ with respect to $H$, i.e.,
$$
0\not\in\mathrm{co} \,H_p(x_0,D^+u(x_0))\,,
$$
then it was proved in \cite{Albano_Cannarsa2002} that there exists a nonconstant singular arc $\mathbf{x}$ from $x_0$ which is a generalized characteristic.

The study of the local propagation of singularities along generalized characteristics was later refined in \cite{Yu2006} and \cite{Cannarsa_Yu2009}. For weak KAM solutions, local propagation results were obtained in \cite{Cannarsa_Cheng_Zhang2014}. An interesting interpretation of the above singular curves as part of the flow of fluid particles has been recently proposed in \cite{Khanin_Sobolevski2016}  (see also \cite{Stromberg_Ahmadzadeh2014} for related results).

For many reasons, we will pay more attention to the global propagation results, i.e., if the initial point $x_0$ belongs to $\SING$, any solution $\mathbf{x}$ of \eqref{eq:gc_intro} can be extended to $(0,+\infty]$ such that $\mathbf{x}(t)\in\SING$ for all $t>0$. Except for some global propagation results by energy-estimates methods (see, for instance, \cite{ACNS2013,Cannarsa_Mazzola_Sinestrari2015,Albano2016_1}), for weak KAM solutions of equation \eqref{eq:weak_KAM} on the whole space, a global propagation result was obtained in \cite{Cannarsa_Cheng3} by an intrinsic approach under much more general conditions.

Indeed, in \cite{Cannarsa_Cheng3}, the problem was solved by using the positive type Lax-Oleinik semigroup which gives an intrinsic explanation of the propagation of singularities only according to the associated system of characteristics. Later, in \cite{Cannarsa_Cheng_Fathi2017}, the method was applied to obtain  topological results for $\CUT$, the cut locus with respect to $u$, and $\SING$ such as the homotopy equivalence between the complement of the projected Aubry set of $u$ and $\CUT$ or $\SING$, and the local path-connectedness of $\CUT$ and $\SING$. Recently, in \cite{Cannarsa_Chen_Cheng2018}, for mechanical systems on the torus, the relation of the $\omega$-limit sets of the semi-dynamics of generalized characteristics with the projected Aubry sets in weak KAM theory was exposed.

In this paper, we will review some recent results on the global propagation and the applications based the intrinsic method above. We also attempt to raise some open problems related to this topic, from analytic, topological, geometric and dynamic points of view respectively.

The paper is organized as follows. In section 2, we review some basic materials for our results. In section 3, some recent results are discussed. The last section is mainly composed of some open problems with explanations.

\medskip

\noindent{\bf Acknowledgments} This work was partially supported by the Natural Scientific Foundation of China (Grant No. 11631006 and No.11790272) and the Istituto Nazionale di Alta Matematica ``Francesco Severi'' (GNAMPA 2017 Research Projects). The authors acknowledge the MIUR Excellence Department Project awarded to the Department of Mathematics, University of Rome Tor Vergata, CUP E83C18000100006.
\section{Preliminaries}

In this section, we want to give a brief review of the Tonelli theory in calculus of variations, weak KAM theory, and basic facts from non-smooth analysis we need. %For convenience we always suppose $M=\R^n$ except for some specific situation.

\subsection{Semiconcave functions}

Let $\Omega\subset\R^n$ be a convex set. We recall that a function $u:\Omega\rightarrow\R$ is said to be {\em semiconcave} (with linear modulus) if there exists a constant $C>0$ such that
\begin{equation}\label{eq:SCC}
\lambda u(x)+(1-\lambda)u(y)-u(\lambda x+(1-\lambda)y)\leqslant\frac C2\lambda(1-\lambda)|x-y|^2
\end{equation}
for any $x,y\in\Omega$ and $\lambda\in[0,1]$.  Any constant $C$ that satisfies the above inequality  is called a {\em semiconcavity constant} for $u$ in $\Omega$. A function $u:\Omega\rightarrow\R$ is said to be {\em semiconvex} if $-u$ is semiconcave.

%When $u:\Omega\to\R$ is continuous, it can be proved that $u$ is semiconcave with constant $C$ if and only if
%$$
%u(x)+u(y)-2u\left(\frac{x+y}2\right)\leqslant \frac C2|x-y|^2
%$$
%for any $x,y\in\Omega$.
%
%Hereafter, we assume that $\Omega$ is a nonempty open convex subset of $\R^n$.

We recall that a function $u:\Omega\rightarrow\R$ is said to be {\em locally semiconcave} (resp. {\em locally semiconvex}) if for each $x\in\Omega$ there exists an open ball $B(x,r)\subset\Omega$ such that $u$ is a semiconcave (resp. semiconvex) function on $B(x,r)$.

Let $u:\Omega\subset\R^n\to\R$ be a continuous function. We recall that, for any $x\in\Omega$, the closed convex sets
\begin{align*}
D^-u(x)&=\left\{p\in\R^n:\liminf_{y\to x}\frac{u(y)-u(x)-\langle p,y-x\rangle}{|y-x|}\geqslant 0\right\},\\
D^+u(x)&=\left\{p\in\R^n:\limsup_{y\to x}\frac{u(y)-u(x)-\langle p,y-x\rangle}{|y-x|}\leqslant 0\right\}.
\end{align*}
are called the {\em (Dini) subdifferential} and {\em superdifferential} of $u$ at $x$, respectively.

Let now $u:\Omega\to\R$ be locally Lipschitz. We recall that a vector $p\in\R^n$ is said to be a {\em reachable} (or {\em limiting}) {\em gradient} of $u$ at $x$ if there exists a sequence $\{x_n\}\subset\Omega\setminus\{x\}$, converging to $x$, such that $u$ is differentiable at $x_k$ for each $k\in\N$ and
$$
\lim_{k\to\infty}Du(x_k)=p.
$$
The set of all reachable gradients of $u$ at $x$ is denoted by $D^{\ast}u(x)$. For the facts on the superdifferential of a semiconcave function on $\Omega\subset\R^n$, the readers can refer to \cite{Cannarsa_Sinestrari_book}.

A point $x\in\Omega$ is called a {\em singular point} of $u$ if $D^+u(x)$ is not a singleton. The set of all singular points of $u$, also called the {\em singular set} of $u$, is denoted by $\SING$.

For a systematic treatment of semiconcavity and its applications to Hamilton-Jacobi equations, calculus of variations and optimal controls, the readers can refer to the monograph \cite{Cannarsa_Sinestrari_book} (see, also, \cite{Villani_book}).

\subsection{Tonelli Theory and regularity properties of the fundamental solutions}

For the reason that our main purpose is to adapt a finer analysis of Tonelli theory, we will neglect various relevant problem under rather general conditions (see, for instance, \cite{Clarke_Vinter1985_1}). We begin with a classical setting of autonomous Lagrangians. A function $L\in C^2(\R^n\times\R^n,\R)$ is called a {\em Tonelli Lagrangian}\footnote{The collection of conditions (L1) -(L3) is exactly the classical Tonelli conditions when $M$ is a compact manifold.} if $L$ satisfies the following conditions:
\begin{enumerate}[(L1)]
\item {\em Uniform convexity}: There exists a nonincreasing function $\nu:[0,+\infty)\to(0,+\infty)$ such that
\begin{equation*}
L_{vv}(x,v)\geqslant \nu(|v|)I\qquad\forall  (x,v)\in\R^n\times \R^n.
\end{equation*}
\item {\em Growth condition}: There exist two superlinear function $\theta_1,\theta_2:[0,+\infty)\to[0,+\infty)$ and a constant $c_0>0$ such that
$$
\theta_2(|v|)\geqslant L(x,v)\geqslant\theta_1(|v|)-c_0\qquad\forall  (x,v)\in\R^n\times \R^n.
$$
\item {\em Uniform regularity}: There exists a nondecreasing function $K:[0,+\infty)\to[0,+\infty)$ such that, for every multi-index $|\alpha|=1,2$,
\begin{equation*}
|D^\alpha L(x,v)|\leqslant K(|v|)\qquad\forall  (x,v)\in\R^n\times \R^n,
\end{equation*}
\end{enumerate}
If $L$ is a Tonelli Lagrangian, the associated Hamiltonian $H$, which is called a {\em Tonelli Hamiltonian}, is the Fenchel-Legendre dual of $L$ defined by
\begin{equation}\label{eq:H}
H(x,p)=\sup_{v\in\R^n}\big\{\langle p,v \rangle-L(x,v)\big\}\qquad(x,p)\in \R^n\times\R^n\,.
\end{equation}
%The Hamiltonian $H$ satisfies similar structural properties which will be referred to as (H1)-(H3). For this reason, $H$ is called a {\em Tonelli Hamiltonian}.

%\subsection{Fundamental solutions, Lax-Oleinik semigroups and Weak KAM solutions}

For any $t>0$, given $x, y\in\R^n$, we set
$$
\Gamma^t_{x,y}=\{{\xi\in W^{1,1}([0,t];\R^n): \xi(0)=x,\xi(t)=y}\}
$$
and define
\begin{equation}\label{fundamental_solution}
A_{t}(x,y)=\min_{\xi\in\Gamma^t_{x,y}}\int^{t}_{0}L(\xi(s),\dot{\xi}(s))ds.
\end{equation}
The existence of the minimum above is a well-known result in Tonelli's theory (see, for instance, \cite{Fathi_book,Buttazzo_Giaquinta_Hildebrandt_book}). Any $\xi\in\Gamma^t_{x,y}$ at which the minimum in \eqref{fundamental_solution} is achieved  is called a {\em minimizer} for $A_t(x,y)$. Such a minimizer $\xi$ is of class $C^2$ by classical results. In the PDE literature, $A_t(x,y)$ is also called the {\em fundamental solution} of the associated Hamilton-Jacobi equations in the form \eqref{eq:evolutionary_intro} (see, for instance, \cite{McEneaney_Dower2015}).

Now, we collect some basic regularity results on the fundamental solution $A_t(x,y)$. The readers can refer to \cite{Cannarsa_Cheng3} for  detailed proofs.

\begin{Pro}[Semiconcavity of the fundamental solution]\label{semiconcave_A_t}
Suppose $L$ is a Tonelli Lagrangian. Then for any $\lambda>0$ there exists a constant $C_\lambda>0$ such that for any $x\in\R^n$, $t\in(0,2/3)$, $y\in B(x,\lambda t)$, and $(h,z)\in\R\times\R^n$  satisfying $|h|<t/2$ and $|z|<\lambda t$ we have
\begin{equation}\label{eq:seminconcavity_A_t}
A_{t+h}(x,y+z)+A_{t-h}(x,y-z)-2A_t(x,y)\leqslant\frac {C_\lambda} t\big(|h|^2+|z|^2\big).
\end{equation}
Consequently, $(t,y)\mapsto A_t(x,y)$ is locally semiconcave in $(0,1)\times\R^n$, uniformly with respect to $x$.
\end{Pro}

In the general case $t>0$, a local semiconcavity result holds true for $A_t(x,y)$ in the same form as \eqref{eq:seminconcavity_A_t} with $C_\lambda$ depending on $t$.

\begin{Pro}[Covexity of the fundamental solution for small time]\label{convexity_A_t}
Suppose $L$ is a Tonelli Lagrangian and, for any $\lambda>0$,  there exists $t_\lambda'>0$ such that, for any $x\in\R^n$ the function $(t,y)\mapsto A_t(x,y)$ is semiconvex on the  cone
\begin{equation}
\label{cone}
S_\lambda(x,t_\lambda'):=\big\{(t,y)\in\R\times\R^n~:~0<t< t_\lambda',\; |y-x|<\lambda t\big\}\,,
\end{equation}
and there exists a constant $C''_\lambda>0$ such that for all $(t,y)\in S_\lambda(x,t_\lambda')$,  all $h\in[0,t/2)$, and  all $z\in  B(0,\lambda t)$ we have that
\begin{equation}\label{eq:semiconvexity}
A_{t+h}(x,y+z)+A_{t-h}(x,y-z)-2A_t(x,y)\geqslant - \frac{C''_ \lambda}{t}(h^2+|z|^2).
\end{equation}
Moreover, there exist $t''_\lambda\in(0,t_\lambda']$ and  $C'''_{\lambda}>0$ such that for all $t\in(0,t''_\lambda]$ the function $A_t(x,\cdot)$ is uniformly convex on $B(x,\lambda t)$ and   for all $y\in B(x,\lambda t)$ and   $z\in  B(0,\lambda t)$ we have that
\begin{equation}\label{eq:convexity_local}
A_{t}(x,y+z)+A_{t}(x,y-z)-2A_t(x,y)\geqslant \frac{C'''_{\lambda}}{t}|z|^2.
\end{equation}
\end{Pro}

The combination of the two propositions above leads to a $C^{1,1}$ result which is essentially connected to the Lasry-Lions type regularization (\cite{Lasry_Lions1986,Bernard2007}).

\begin{Pro}\label{C11_A_t}
Suppose $L$ is a Tonelli Lagrangian and, for any $\lambda>0$,  there exists $t_\lambda'>0$ such that, for any $x\in\R^n$ the functions $(t,y)\mapsto A_t(x,y)$ and $(t,y)\mapsto A_t(y,x)$ are of class $C^{1,1}_{\text{loc}}$ on the cone $S_{\lambda}(x,t_\lambda')$ defined in \eqref{cone}.
Moreover, for all $(t,y)\in S(x,t_\lambda')$
\begin{align}
D_yA_t(x,y)=&L_v(\xi(t),\dot{\xi}(t)),\label{eq:diff_A_t_y}\\
D_xA_t(x,y)=&-L_v(\xi(0),\dot{\xi}(0)),\label{eq:diff_A_t_x}
\end{align}
where $\xi\in\Gamma^t_{x,y}$ is the unique minimizer for $A_t(x,y)$.
\end{Pro}

\subsection{Lax-Oleinik semigroup and weak KAM solution}
We will concentrate the case when $M=\R^n$. For all  $(t,x)\in\R_+\times\R^n$  let
\begin{equation}
\label{eq:phipsi}
\phi^x_t(y)=u(y)-A_t(x,y)\quad\mbox{and}\quad \psi^x_t(y) =u(y)+A_t(y,x)\qquad(y\in\R^n)
\end{equation}
where $A_t$ is the fundamental solution defined in \eqref{fundamental_solution}. The Lax-Oleinik operators $T^-_t$ and $T^+_t$ are defined as follows
\begin{gather}
T^+_tu(x)=\sup_{y\in\R^n}\phi^x_t(y),\quad x\in\R^n,
\label{L-L regularity_sup}\\
T^-_tu(x)=\inf_{y\in\R^n}\psi^x_t(y),\quad x\in\R^n.
\label{L-L regularity_inf}
\end{gather}
The functions $\phi^x_t$ and $\psi^x_t$ are called {\em local barrier functions}.

We define the Ma\~n\'e's critical value as
$$
c_H[0]=-\inf\{A_t(x,x)/t: t>0, x\in\R^n\}.
$$
A continuos function $u\in C(\R^n,\R)$ is called a {\em weak KAM solution} of the Hamiltoni-Jacobi equation
\begin{equation}\label{eq:weak_KAM}
	H(x,Du(x))=c_H[0],
\end{equation}
if and only if $T^-_tu=u$ for all $t>0$, where $T^-_t$ is the negative type Lax-Oleinik operator with respect to $L+c_H[0]$. %The constant $c_H[0]$ is called {\em Ma\~n\'e critical with respect to $H$}. Without loss of generality, we suppose $c_H[0]=0$.

We say a function $u:\R^n\to\R$ is dominated by $L+c_H[0]$, denoted by $u\prec L+c_H[0]$, if for any curve $\gamma\in\Gamma^t_{x,y}$,
$$
u(y)-u(x)\leqslant\int^t_0L(\gamma(s),\dot{\gamma}(s))\ ds+c_H[0]t.
$$
It is well known that $u\prec L+c_H[0]$ if and only if $u$ is a viscosity subsolution of \eqref{eq:weak_KAM}. If $u\prec L+c_H[0]$, $a<b$, a curve $\gamma:[a,b]\to\R^n$ is called {\em $(u,L,c_H[0])$-calibrated} or {\em $u$-calibrated} in short if
$$
u(\gamma(b))-u(\gamma(a))=\int^b_aL(\gamma(s),\dot{\gamma}(s))\ ds+c_H[0](b-a).
$$
A curve $\gamma$ defined from $(-\infty,t_0]$ (resp. $[t_0,+\infty)$, $(-\infty,+\infty)$) to $\R^n$ is called a {\em backward} (resp. {\em forward, global}) {\em $u$-calibrating curve} if $\gamma\vert_{[a,b]}$ is $u$-calibrated for any $a<b\leqslant t_0$ (resp. $t_0\leqslant a<b$, $a<b$).

We also define the Aubry set with respect to a weak KAM solution $u$ of \eqref{eq:weak_KAM} as
$$
\IU=\{x\in\R^n: \text{$x=\gamma(0)$ with $\gamma:\R\to\R^n$ a global $u$-calibrated curve}\}.
$$
If $M$ is a compact manifold, in view of classical weak KAM theory, the projected Aubry set $\mathcal{A}=\bigcap_{u}\IU$ is nonempty and compact, where the intersection is taken over all the weak KAM solutions $u$'s of \eqref{eq:weak_KAM}.

\begin{Pro}[\cite{Fathi_Maderna2007,Cannarsa_Cheng3,Chen_Cheng2016}]\label{weak_KAM}
Let $M=\R^n$. Then, there exists a unique constant $c_H[0]$ such that \eqref{eq:weak_KAM} admit a weak KAM solution. In addition, each weak KAM solution of \eqref{eq:weak_KAM} is globally Lipschitz and semiconcave.

Moreover, let $\phi^x_t$ and $\psi^x_t$ be the local barrier functions with respect to $L+c_H[0]$. Then, there exists a  constant $\lambda_0>0$, depending only on $L$ and ${\rm Lip} (u)$, such that, for any $(t,x)\in\R_+\times\R^n$ and any maximum (resp. minimum) point $y_{t,x}$ (resp. $z_{t,x}$) of $\phi^x_t$ (resp. $\psi^x_t$), we have
\begin{equation}
\label{eq:maxbound}
	|y_{t,x}-x|\leqslant \lambda_0t\,\quad (\text{resp.}\ |z_{t,x}-x|\leqslant \lambda_0t).
\end{equation}

In addtion, there exists a one-to-one correspondence between $p\in D^*u(x)$ and the global minimizers $z_{t,x}$ of $\psi^x_t$ for all $t>0$.
\end{Pro}

We also introduce the set $\CUT$ of cut points of $u$, as the set of points $x\in\R^n$ where no backward $u$-calibrating curve ending at $x$ can be extended to a $u$-calibrating curve beyond $x$. Equivalently, if $\gamma:[a,b]\to\R^n$ is a $u$-calibrating curve with $x\in\gamma([a,b])$ then $x=\gamma(b)$.

At the end of this section, we introduce the concept of generalized characteristic first introduced in \cite{Albano_Cannarsa2002}. A Lipschitz arc $\mathbf{x}:[0,T]\to\R^n,\,(T>0),$ is said to be a {\em generalized characteristic} of the Hamilton-Jacobi equation \eqref{eq:weak_KAM} if $\mathbf{x}$ satisfies the differential inclusion
\begin{equation}\label{generalized_characteristics}
\dot{\mathbf{x}}(s)\in\mathrm{co}\, H_p\big(\mathbf{x}(s),D^+u(\mathbf{x}(s))\big),\quad \text{a.e.}\;s\in[0,T]\,.
\end{equation}
We say that $x\in\R^n$  is a {\em critical point} with respect to a viscosity solution  $u$ of \eqref{eq:weak_KAM} if
\begin{equation}\label{eq:critical1}
	0\in \text{co}\,H_p(x,D^+u(x)).
\end{equation}
%and a {\em strong critical point} with respect to $u$  if
%\begin{equation}\label{eq:critical2}
%	0\in H_p(x,D^+u(x)).
%\end{equation}

\section{Intrinsic approach of global propagation of singularities}

Unlike an energy-estimates method used in the global propagation of singularities of the viscosity solutions developed in \cite{ACNS2013,Albano2016_1,Cannarsa_Mazzola_Sinestrari2015} for eikonal equations and mechanical systems, an intrinsic approach first appears in \cite{Cannarsa_Cheng3} which is only based on the associated characteristic systems and can be applied to rather general problems.

\subsection{Global propagation of singularities}

We begin with the setting from \cite{Cannarsa_Cheng3}. The following proposition is a crucial step for the theory. Now, we give a proof which is different from the original one in \cite{Cannarsa_Cheng3}.

\begin{Pro}[\cite{Cannarsa_Cheng3}]\label{maximizer_singular}
Let $L$ be a Tonelli Lagrangian and let $H$ be the associated Hamiltonian. Suppose $u:\R^n\to\R$ is a Lipschitz continuous semiconcave viscosity solution of \eqref{eq:weak_KAM}. If $x\in\CUT$, then any maximizer of $\phi^x_t$ is contained in $\SING$ for all $t>0$. Moreover, there exists $t_0\in (0,1]$ such that, for all  $(t,x)\in(0,t_0]\times\R^n$, $\phi^x_t$ admits a unique maximum point $y_{t,x}$ and the curve
\begin{equation}\label{maximizer_arc}
\mathbf{y}(t)=\mathbf{y}_x(t):=\begin{cases}
x&\mbox{if}\quad t=0\\
y_{t,x}&\mbox{if}\quad t\in(0,t_0]
\end{cases}
\end{equation}
satisfies $\lim_{t\to0}\mathbf{y}(t)=x$.
\end{Pro}

\begin{proof}%[Proof of Proposition \ref{maximizer_singular}]
	For any $t>0$ and $y_{t,x}\in\arg\max\phi^x_t$ (which is nonempty by Proposition \ref{weak_KAM}), suppose $y_{t,x}$ is a differentiable point of $u$. Thus
	$$
	0\in D^+\{u(\cdot)-A_t(x,\cdot)\}(y_{t,x})=Du(y_{t,x})-D^-\{A_t(x,\cdot)\}(y_{t,x}),
	$$
	equivalently, $Du(y_{t,x})\in D^-\{A_t(x,\cdot)\}(y_{t,x})$. It follows that $A_t(x,\cdot)$ is differentiable at $y_{t,x}$ and
	$$
	p_{t,x}=Du(y_{t,x})=D_yA_t(x,y_{t,x})
	$$
	since $A_t(x,\cdot)$ is locally semiconcave. Therefore, there exists two $C^2$ extremals (with respect to the associated Euler-Lagrange equation) $\xi_{t,x}:[0,t]\to\R^n$ and $\gamma_x:(-\infty,t]\to\R^n$ such that $\xi_{t,x}(0)=x$, $\gamma_x(t)=\xi_{t,x}(t)=y_{t,x}$ and $p_{t,x}=L_v(\gamma_x(t),\dot{\gamma}_x(t))=L_v(\xi_{t,x}(t),\dot{\xi}_{t,x}(t))$. Since $\xi_{t,x}$ and $\gamma_x$ has the same endpoint condition at $t$, then they coincide on $[0,t]$. This leads to a contradiction since $x=\gamma_x(0)$ and $\gamma_x(0)\not\in\CUT$.
	
	Now we turn to the proof of the last part of the proposition. Let  $C_1>0$ be a semiconcavity constant for $u$ on $\R^n$ and let $\lambda_0$ be the positive constant in Proposition \ref{weak_KAM}. By  Proposition~\ref{convexity_A_t} with $\lambda=1+\lambda_0$, we deduce that there exists $t_0\in(0,1]$ and a constant $C_2>0$ such that for every $(t,x)\in (0,t_0]\times\R^n$, every $y\in B(x,\lambda t)$, and every $z\in B(0,\lambda t)$ we have that
	\begin{equation*}
	A_{t}(x,y+z)+A_{t}(x,y-z)-2A_t(x,y)\geqslant \frac{C_2}{t}|z|^2.
	\end{equation*}
	Thus,  $\phi^x_t(y)=u(y)-A_t(x,y)$ is strictly concave on $\overline{B}(x,\lambda t)$ for all $t\in(0,t_0]$ provided that we further restrict $t_0$ in order to have
	\begin{equation}\label{eq:t_0}
	t_0<\frac{C_2}{C_1}.
	\end{equation}
	Then, for all such numbers $t$, there exists a unique maximum point $y_{t,x}$ of $\phi^x_t$ in $\overline{B}(x,\lambda t)$. In fact, $y_{t,x}$ is an interior point of $B(x,\lambda t)$ since, by Proposition \ref{weak_KAM}, we have that $|y_t-x|\leqslant\lambda_0 t$.
\end{proof}

We can see that the curve $\mathbf{y}$ defined on $[0,t_0]$ as in \eqref{maximizer_arc} is indeed a generalized characteristic which can be extend to $+\infty$ since $t_0$ is independent of the initial point $x$.

\begin{Pro}[\cite{Cannarsa_Cheng3}]\label{Main_lemma_g_c}
Let $L$ be a Tonelli Lagrangian. Let $t_0\in(0,1]$ be given by Proposition \ref{maximizer_singular}. For any fixed $x\in\R^n$, let $\mathbf{y}:[0,t_0]\to\R^n$ be the curve defined in \eqref{maximizer_arc}. Then
\begin{enumerate}[\rm (a)]
  \item $\mathbf{y}$ is Lipschitz on $[0,t_0]$.
  \item The singular arc $\mathbf{y}:[0,t_0]\to\R^n$ is a generalized characteristic, that is,
  \begin{equation*}%\label{eq:sing_gen_char}
  \dot{\mathbf{y}}(\tau)\in\mbox{\rm co}\, H_p(\mathbf{y}(\tau),D^+u(\mathbf{y}(\tau))),\quad\text{a.e.}\ \tau\in[0,t_0].
  \end{equation*}
  Moreover,
  \begin{equation*}%\label{eq:strong_gen_char_1}
  \dot{\mathbf{y}}^+(0)=H_p(x,p_0),
  \end{equation*}
  where $p_0$ is the unique element of minimal energy:
  $$
  H(x,p)\geqslant H(x,p_0),\quad \forall p\in D^+u(x).
  $$
%  \item $\mathbf{y}(t)=x$ if and only if $D_yA_t(x,x)\in D^+u(x)$.
%  \item If $x$ is not a strong critical point of $u$ then there exists $t\in(0,t_0]$ such that $\mathbf{y}(s)\not=x$ for all $s\in(0,t]$.
\end{enumerate}
\end{Pro}

\begin{Rem}\label{rem:manifold}
	It is worth pointing out that the idea and the technique used here for the global propagation results are still valid for manifold case because of the local nature. So, by using local chart, one can adapt to any compact manifold $M$ definitely. For a rigorous treatment of semiconcavity on manifold using local charts, the readers can refer to \cite{Fathi_Figalli2010}.
\end{Rem}

\subsection{Topology of $\CUT$ and $\SING$}\label{sec:topology}

In this section, we suppose that $M$ is a $C^2$ closed manifold and $L$ is a Tonelli Lagrangian. Notice the fact that $\SING\subset\CUT\subset M\setminus\IU$, and $\SING\subset\CUT\subset\BSING$.

\begin{Lem}[\cite{Cannarsa_Cheng_Fathi2017}]\label{main_lemma}
There exists some $t>0$, and a (continuous) homotopy $F:
M\times [0, t]\to M$, with the following properties:
\begin{enumerate}[\rm (a)]
  \item for all $x\in M$, we have $F(x,0)=x$;
  \item if $F(x,s)\not\in\SING$, for some $s>0$, and $x\in M$, then the curve $\sigma\mapsto F(x,\sigma)$ is $u$-calibrating on $[0,s]$;
  \item if there exists a $u$-calibrating curve $\gamma:[0,s]\to M$, with $\gamma(0)=x$, then $\sigma\mapsto F(x,\sigma)=\gamma(\sigma)$, for every $\sigma\in [0,\min(s,t)]$.
\end{enumerate}
\end{Lem}

Lemma \ref{main_lemma} claims that $\mathbf{y}_x(t)$ defined in \eqref{maximizer_arc} establishs a continuous homotopy on $M$, which leads to the following topological properties of $\CUT$ and $\SING$ in the homotopy sense.

\begin{Pro}[\cite{Cannarsa_Cheng_Fathi2017}]\label{homotopy_1}
The inclusion $\Sigma(u)\subset\CUT\subset\BSING\cap(M\setminus\IU)\subset M\setminus\IU$ are all homotopy equivalences.	
\end{Pro}

The result above can be regarded as an extension of the main result in \cite{ACNS2013} to the context of weak KAM theory.

\begin{Pro}[\cite{Cannarsa_Cheng_Fathi2017}]\label{homotopy_2}
The spaces $\SING$, and $\CUT$ are locally contractible, i.e. for every $x\in\SING$ (resp. $x\in\CUT$) and every neighborhood $V$ of $x$ in $\SING$ (resp. $\CUT$), we can find a neighborhood $W$ of $x$ in $\SING$ (resp. $\CUT$), such that $W\subset V$, and $W$ in null-homotopic in $V$.

Therefore $\SING$, and $\CUT$ are locally path connected.	
\end{Pro}

\subsection{Dynamics of generalized characteristic semi-flow}\label{sec:dynamic}

In this section, we have decided to concentrate on the important example of mechanical systems on the torus $\T^n$. The main reason is that the uniqueness of the solution to \eqref{generalized_characteristics} is not guaranteed for general Hamiltonians. So, the associated semiflow may fail to be well defined. But, such a semiflow is well defined for mechanical systems. We say a Hamiltonian $H$ {\em has the uniqueness property} if there exists a unique generalized characteristic staring from any given initial point.

Let us consider the {\em mechanical Hamiltonian} in the following form:
\begin{equation}\label{eq:mech_intro}
	H(x,p)=\frac 12\langle A(x)p,p\rangle+V(x),\quad x\in\R^n, p\in\R^n
\end{equation}
where $x\mapsto A(x)$ is a $\T^n$-periodic $C^2$-smooth map taking values in the real space of $n\times n$ positive definite symmetric matrices, and $V$ is a $\T^n$-periodic function on $\R^n$, at least of class $C^2$, satisfying $\max_{x\in\R^n}V(x)=0$. A typical Hamiltonian $H$ having the uniqueness property is a mechanical Hamiltonian as in \eqref{eq:mech_intro}.

For any $c\in\R^n$, let
\begin{align*}
	H^c(x,p):=H(x,c+p)-\alpha_H(c),\quad (x,p)\in\T^n\times(\R^n)^*,
\end{align*}
where $\alpha_H(\cdot)$ is Mather's $\alpha$-function. Let $u_c$ be a $\T^n$-periodic weak KAM solution of the Hamilton-Jacobi equation
\begin{equation}\label{intr_HJ}
	H^c(x,Du_c(x))=0,\quad x\in \T^n.
\end{equation}
Here we also looks $u_c$ as a real-valued function on $\T^n$. We also define
\begin{equation}\label{intro:v}
v_c(x):=\langle c,x\rangle+u_c(x),\quad x\in\R^n.
\end{equation}

The following proposition, based on the results in \cite{Cannarsa_Cheng_Zhang2014,Cannarsa_Cheng3,Cannarsa_Yu2009,Albano_Cannarsa2002}, is a collection of properties of the generalized characteristics associated with the pair $\{H,v_c\}$ with $H$ as in \eqref{eq:mech_intro}.

\begin{Pro}\label{properties_g_c}
Let $H$ be a $\T^n$-periodic Hamiltonian as in \eqref{eq:mech_intro}, let $c\in\R^n$, and let $u_c$ be a viscosity solution of  equation \eqref{intr_HJ}. Then $v_c(x)$ introduced in \eqref{intro:v} has the following properties. Fix any $x\in\R^n$.
\begin{enumerate}[\rm (a)]
  \item There is a unique Lipschitz arc $\mathbf{x}:[0,+\infty)\to\R^n$ such that
\begin{equation}
\label{eq:genchar}
\dot{\mathbf{x}}(s)\in A(\mathbf{x}(s))D^+v_c(\mathbf{x}(s))
\end{equation}
and $\mathbf{x}(0)=x$. Moreover, denoting by $\mathbf{y}$ the solution  of \eqref{eq:genchar} starting from any other point $y\in\R^n$, we have that
	$$
	|\mathbf{x}(s)-\mathbf{y}(s)|\leqslant C|x-y|, \quad s\in[0,\tau]
	$$
for some constant $C\geqslant 0$ depending on $\tau>0$.
\item If $x\in\mbox{\rm Sing}\, (v_c)$, then $\mathbf{x}(s)\in\mbox{\rm Sing}\, (v_c)$ for all $s\in [0,+\infty)$.	
   \item The right derivative $\dot{\mathbf{x}}^+(s)$ exists for all $s\in[0,+\infty)$ and
\begin{equation*}
\label{generalized_characteristics_mech_sys}
\dot{\mathbf{x}}^+(s)=A(\mathbf{x}(s))p(s),\qquad\forall s\in[0,+\infty),
\end{equation*}
where $p(s)$ is the unique point of $D^+v_c(\mathbf{x}(s))$ such that
\begin{equation}\label{minimality}
\langle A(\mathbf{x}(s))p(s),p(s)\rangle=\min_{p\in D^+v_c(\mathbf{x}(s))}\langle A(\mathbf{x}(s))p,p\rangle.
\end{equation}
Moreover, $\dot{\mathbf{x}}^+(s)$ is right-continuous.
   \item The right derivative of $v_c(\mathbf{x}(\cdot))$ has the following representation:
   \begin{equation}
   	\frac d{ds^+}v_c(\mathbf{x}(s))=\langle p(s),A(\mathbf{x}(s))p(s)\rangle,\quad s\in[0,+\infty).
   \end{equation}
   \item If $\alpha(c)>0$  and $\dot{\mathbf{x}}^+(s)\not=0$ for all $s\in [0,\tau]$, then
\begin{equation}\label{eq:increasing}
v(\mathbf{x}(s_1))<v(\mathbf{x}(s_2)),\quad\ \text{for all}\quad 0\leqslant s_1<s_2\leqslant\tau.
\end{equation}
   \item If $x$ is not a critical point of $v_c$, then $\mathbf{x}$ is injective on $[0,\tau]$ for some $\tau>0$.
   \item If $\alpha(c)>0$, then all the critical points of $v_c$ are contained in $\mbox{\rm Sing}\,(v_c)$.
 \end{enumerate}
\end{Pro}

For any $x\in\R^n$, we denote by $\mathbf{x}(\cdot,0,x)$ the unique generalized characteristic starting from $x$. By Proposition \ref{properties_g_c} (a) and (g), it is clear that
\begin{equation}\label{semiflow_noncmpt}
	\Phi_t(x):=\mathbf{x}(t,0,x),\quad (t,x)\in[0,\infty)\times\R^n
\end{equation}
defines a semiflow on $\R^n$. It is not difficult to see, one can also introduce the semiflow on $\T^n$ in view of Proposition \ref{maximizer_singular} and Remark \ref{rem:manifold}. We will abuse the notation $\Phi_t(x)$ for both the semiflow on $\R^n$ and $\T^n$.

\begin{Pro}[\cite{Cannarsa_Chen_Cheng2018}]
Let $\Phi_t(x)$ be the semiflow on $\R^n$ defined by the generalized characteristic determined by $v_c$ and let the following regularity condition be satisfied:\\
{\rm (R)} \ the regular values of $v_c$ are dense in $\R$.\\
Then $\mathcal{R}(\Phi_t)\subset\mbox{\rm Crit}\, (v_c)$. In particular, we have
\begin{equation}\label{eq:relation_limit_sets}
	L(\Phi_t):=\mbox{\rm cl}\, \Big(\cup\{\omega(\Phi_t,x): x\in X\}\Big)=\Omega(\Phi_t)=\mathcal{R}(\Phi_t)=\mbox{\rm Crit}\, (v_c),
\end{equation}
where $\Omega(\Phi_t)$ (resp. $\omega(\Phi_t,x)$) is the $\omega$-limit set of the semiflow $\Phi_t$ (resp. the semi-orbit $\Phi_t(x)$), and $\mathcal{R}(\Phi_t)$ is the chain-recurrent set of $\Phi_t$.
\end{Pro}

For the definition of various kinds of invariant set of the semiflow $\Phi_t(x)$, the readers can refer to \cite{Robinson_book1999,Katok_Hasselblatt_book,Hurley1995}.

In fact, there is another associated semiflow on $\T^n$ defined by
\begin{equation}\label{semiflow_cmpt}
	\phi_t(x):[0,+\infty)\times\T^n\to\T^n \quad\text{and}\quad \phi_t\circ \pi=\pi\circ\Phi_t,
\end{equation}
where $\pi:\R^n\to\T^n$ is the canonical projection and $\Phi_t$ is defined in \eqref{semiflow_noncmpt}.

In view of Proposition \ref{maximizer_singular} and Remark \ref{rem:manifold}, the unique global generalized characteristics $\mathbf{x}=\mathbf{x}(t,0,x):[0,+\infty)\to\T^n$ defined in Proposition \ref{properties_g_c} can be determined inductively as
\begin{equation}\label{eq:rescaled_GC}
	\mathbf{x}(t,0,x)=\phi_{t-it_0}(\mathbf{x}(it_0,0,x)),\quad\forall t\in[it_0,(i+1)t_0],\,i\in\N.
\end{equation}
Owing to the uniqueness of generalized characteristics when $H$ has the form \eqref{eq:mech_intro}, for any $0<\tau\leqslant t_0$ we also have
	\begin{equation}\label{eq:using_tau}
		\phi_t(x)=\mathbf{x}(t,0,x)=\mathbf{y}_{\mathbf{x}(i\tau,0,x)}(t-i\tau),\quad\forall t\in[i\tau,(i+1)\tau],\,i\in\N.
	\end{equation}

For any $\tau\in(0,t_0]$, let $z^{\tau}_i=\mathbf{x}(i\tau)$, $i\in\N$, and let $\mathcal{Z}^{\tau}$ be the set of all convergent subsequences of $\{z^{\tau}_i\}$. For any strictly increasing sequence of natural numbers $\sigma=\{i_1,i_2,\ldots,i_k,\ldots\}$ and the associated convergent subsequence $z^{\tau}_{\sigma}=\{z^{\tau}_{i_k}\}$, we define
	$$
	N^{\tau}_{\sigma}=\sup\{i_{k+1}-i_k: z^{\tau}_{\sigma}\in\mathcal{Z}^{\tau}\}.
	$$
	If $\tau=t_0$, we will take out the superscript $\tau$ for abbreviation. It is clear $\mathbf{x}$ is unique, then it does not depend on the choice of $\tau\in(0,t_0]$ by \eqref{eq:using_tau}.
	
We denote the $\omega$-limit set of $\mathbf{x}$ by $\omega(\mathbf{x})$. It is clear that $\omega(\mathbf{x})\subset\overline{C(x)}$ where $C(x)$ is the connected component of $\SING$ containing $x$. For each $z\in\omega(\mathbf{x})$, there exists $\sigma$ such that the sequence $z_{\sigma}$ converges to $z$.

\begin{Pro}[\cite{Cannarsa_Chen_Cheng2018}]\label{asymptotic}
	Let $H$ be a mechanical Hamiltonian as in \eqref{eq:mech_intro} and let $u_c$ be a weak KAM solution of \eqref{intr_HJ}. Suppose $x\in\SING$, $\mathbf{x}$ is the unique generalized characteristics starting from $x$, and $C(x)$ is the connected component of $\SING$ containing $x$.
	\begin{enumerate}[\rm (a)]
	\item If $\lim_{t\to\infty}\mathbf{x}(t)$ exists, then there exists $z\in\overline{C(x)}$ such that $0\in H^c_p(z,D^+u_c(z))$.
	\item Suppose that $\lim_{t\to\infty}\mathbf{x}(t)$ does not exist and fix any $\tau\in(0,t_0]$. Then for any $z_{\sigma}\in\mathcal{Z}^{\tau}$ such that $N_{\sigma}<\infty$, there exists a closed generalized characteristic contained in $\overline{\{\mathbf{x}(t):t>0\}}\subset\overline{C(x)}$.
	\item Let $\tau_k\to 0^+$ as $k\to\infty$. If for each $k\in\N$, there exists an $\sigma_k$ such that $z^{\tau_k}_{\sigma_k}\in\mathcal{Z}^{\tau_k}$ with $N_{\sigma_k}^{\tau_k}<\infty$, and $\lim_{k\to\infty}\tau_kN_{\sigma_k}^{\tau_k}=0$, then there exists $z\in\overline{C(x)}$ such that $0\in \textup{co}\,H^c_p(z,D^+u_c(z))$.
	\item Fix any $\tau\in(0,t_0]$ and $z_{\sigma}\in\mathcal{Z}^{\tau}$. If $N_{\sigma}(\tau)=\infty$, then there exists a global generalized characteristic $\mathbf{y}:(-\infty,+\infty)\to\T^n$ such that $\{\mathbf{y}(t): t\in\R\}$ is contained in $\overline{\{\mathbf{x}(t):t>0\}}\subset\overline{C(x)}$.
	\end{enumerate}
\end{Pro}

\begin{Pro}[\cite{Cannarsa_Chen_Cheng2018}]
Suppose $H$ is a mechanical Hamiltonian as in \eqref{eq:mech_intro}, $u_c$ is weak KAM solution of \eqref{intr_HJ}, $x\in\SING$, and $C(x)$ is the connected component of $\SING$ containing $x$. Let $\mathbf{x}:[0,+\infty)\to\T^n$ be the unique generalized characteristic staring from $x$. If there is no critical point of $u_c$ with respect to $H$ in $\overline{C(x)}$, then $\lim_{t\to\infty}\mathbf{x}(t)$ does not exists. In addition, for any $\sigma\in \mathcal{Z}^{t_0}$,  the following properties hold:
	\begin{enumerate}[\rm (a)]%\setliststart{2}
	\item If $ N_{\sigma}<\infty$, then there exists a non-constant closed generalized characteristic $\mathbf{y}:[0,T]\to\T^n$, $\mathbf{y}(0)=\mathbf{y}(T)$, contained in $\overline{\{\mathbf{x}(t):t>0\}}\subset\overline{C(x)}$. Moreover, we have either
	\begin{enumerate}[\rm (1)]
	\item $\mathbf{y}:[0,T]\to\T^n$ is a $C^2$ closed regular characteristic contained in $\mathcal{I}_{u_c}$, or
	\item $\mathbf{y}:[0,T]\to\T^n$ is a closed singular generalized characteristic.
	\end{enumerate}
	\item If $N_{\sigma}=\infty$\footnote{We can always suppose that $N_{\sigma}=\infty$ by choosing a suitable subsequence of $\sigma$ if necessary.}, then there exists a global generalized characteristic $\mathbf{y}:\R\to\T^n$ such that $\{\mathbf{y}(t): t\in\R\}$ is contained in $\overline{\{\mathbf{x}(t):t>0\}}\subset\overline{C(x)}$. Moreover, we have either
	\begin{enumerate}[\rm (1)]
	\item $\mathbf{y}:\R\to\T^n$ is a global singular generalized characteristic, or
	\item $\overline{C(x)}\cap\mathcal{I}_{u_c}\neq\emptyset$. In particular, $\overline{C(x)}\cap\mathcal{A}(H^c)\neq\emptyset$.
	\end{enumerate}
	\end{enumerate}	
\end{Pro}

\begin{Pro}[\cite{Cannarsa_Chen_Cheng2018}]
Suppose $H$ is a mechanical Hamiltonian as in \eqref{eq:mech_intro}, $x\in\SING$ and $\mathbf{x}$ is the unique generalized characteristic starting from $x$. If there exists a point of differentiability of $u_c$ in $\omega(\mathbf{x})$, then $\omega(\mathbf{x})\cap\mathcal{A}(H^c)\not=\varnothing$.	
\end{Pro}

\section{On and beyond propagation of singularities}

The study of propagation of singularities of certain viscosity solutions is closed connected to many topics such as PDE, calculus of variation and optimal controls, Hamiltonian dynamical systems, geometry and so on. Based on the aforementioned results, We will raise some open problems involving propagation of singularities. The list will be much longer than the one in \cite{Chen_Cheng2016}.

\subsection{Analytic aspects}

From the proof of Proposition \ref{maximizer_singular}, we can summarize the basic idea to prove the global propagation of singularities of viscosity solutions governed by generalized characteristics of certain Hamilton-Jacobi equations as follows:
\begin{enumerate}[(i)]
	\item First, we need a representation formulae for the viscosity solutions of certain problems in the form of inf-convolution like what in \eqref{L-L regularity_inf}. That is, such a solution can be regarded as the value functions of an associated problem of calculus of variations and optimal controls.
	\item Second, we need the regularity properties of the associated fundamental solutions such as Proposition \ref{C11_A_t}.
	\item Third, an argument like the proof of Proposition \ref{maximizer_singular} using sup-convolution can be applied to get the result of propagation of singularities.
	\item Finally, we need show that the arc obtained is a generalized characteristic on a time interval $[0,t_0]$ which can be extended to $+\infty$ if we can have some uniform property of $t_0$.
\end{enumerate}

Of course, one should improve certain technique points above to deal with different kinds of problems.

\medskip

\noindent\textbf{A1.} Can technique points (i)-(iv) be applied to various type of problems?

\medskip

For example, in the preprint \cite{CCMW2018}, this method is successfully applied to the Dirichlet problem. Another example is the problem with respect to the Hamilton-Jacobi equations in the form
\begin{equation}\label{eq:HJ_contact}
	H(x,u(x),Du(x))=c,\quad x\in M.
\end{equation}
It is very hopeful to solve the global propagation result for the viscosity solution $u$ of \eqref{eq:HJ_contact} using the program above and recent works on certain contact type Hamilton-Jacobi equations (\cite{Wang_Wang_Yan2017,Wang_Wang_Yan2018,CCWY2018,Wang_Wang_Yan2018_1,Wang_Wang_Yan2018_2}).

\medskip

The only knowledge for us on a Hamiltonian having uniqueness property is a mechanical one. Undoubtedly, this property is not well understood until now. Let $H$ be a mechanical Hamiltonian as in \eqref{eq:mech_intro}, and define
\begin{equation}\label{eq:H_W}
		H_W(x,p)=H(x,p)+\langle W(x),p\rangle,
\end{equation}
where $W$ is a vector field on $M$ with its canonical dual, a differential form $\omega_x$. The Hamiltonian $H_W$ has uniqueness property if $\omega_x$ is a closed 1-form. In general, the distribution (in the sense of Frobenius) $\mathcal{D}=\text{ker}\,\omega_x$ is non-integrable. It is unclear what is the condition to let $H_W$ have uniqueness property.

On the other hand, the differential inclusion \eqref{generalized_characteristics} is also an ordinary differential equation in the sense that all the solutions satisfy the equation almost everywhere. Then, we can also study the uniqueness of generalized characteristic from arbitrary initial points in the Diperna-Lions theory point of view. Is the condition on the  $\text{div}\,W$ is what expected?

\medskip

\noindent\textbf{A2.} What is the essential conditions for a Hamiltonian $H$ having uniqueness property?

\medskip

In our construction of the global generalized characteristics, the singular arc $\mathbf{y}_x(t)$ on $[0,t_0]$ is obtained by Proposition \ref{maximizer_arc} and $t_0$ is independent of the initial point $x$.

\medskip

\noindent\textbf{A3.} Can we drop the uniformness requirement of such $t_0$ to obtain a global result?

\medskip

The uniformness of such $t_0$ holds because of our uniform conditions (L1)-(L3). At least, we hope to only use so called Fathi-Maderna conditions in \cite{Fathi_Maderna2007}.

\medskip

Recalling the alternative approach using Lasry-Lions regularization for the study of the global propagation results in \cite{Cannarsa_Cheng_Fathi2017}, and the results \cite{Chen_Cheng2016,Chen_Cheng_Zhang2018} on the relation between the Lasry-Lions regularization and propagation of singularities, the readers can find that the $C^{1,1}$ regularity result in Proposition \ref{C11_A_t} is a key technique point in the theory. But, when working on a manifold with (smooth) boundary, one may meet the difficulty to deal with the regularity property of the boundary as well as the boundary function (\cite{CCMW2018}).

\medskip

\noindent\textbf{A4.} Can one improve the program (i)-(iv) using a $C^{1,\alpha}$ ($\alpha\in(0,1)$) argument for certain problems involving state constraint?

\medskip

\noindent\textbf{A5.} How about the Lasry-Lions regularization for state constraint type problems?

\medskip

In \cite{Cannarsa_Yu2009}, local results of propagation of singularities involving any semiconcave function $u$ and any Tonelli Hamiltonian $H$ were obtained. In the same paper, the authors also proved the local propagation of singularities governed by a partial differential inclusion of generalized characteristics.

\medskip

\noindent\textbf{A6.} How about the intrinsic nature on the problem of global propagation of singularities on a pair $(u,H)$, especially at a critical point?

\medskip

\noindent\textbf{A7.} How about the intrinsic nature on the partial differential inclusions of generalized characteristics?

\subsection{Dynamic, topological and geomertic aspects}

The main problem in the dynamic aspects of the theory is to exploit the relations between the singular dynamics of generalized characteristics and the regular Hamiltonian dynamics, especially the applications to the Hamiltonian dynamical systems.

As pointed out in Section \ref{sec:dynamic}, an interesting result is the relations between the $\omega$-limit set of the relevant semiflow on $\T^n$, say $\phi_t$, and the Aubry set. It is possible that the singularities of a weak KAM solution evolute along the generalized characteristics approaching the Aubry set.

\medskip

\noindent\textbf{B1.} What is the dynamical nature of the invariant measures produced by the semiflow $\phi_t$?

\medskip

\noindent\textbf{B2.} How about the dynamic and topological structure of the supports of such invariant measures produced by the semiflow $\phi_t$?

\medskip

\noindent\textbf{B3.} Are there some finer properties on $\T^2$?

\medskip

In the study of dynamics of the semiflow $\phi_t$ for mechanical systems, there is an obstacle for the semiflow, i.e., the sets of critical points defined in \eqref{eq:critical1} (see Proposition \ref{asymptotic}). If we concentrate on the mechanical systems, this problem is closed related to the problem of Novikov's critical point theory of closed 1-forms (\cite{Farber_book}). A much more general situation is that the Hamiltonian has the form $H_W$ introduced in \eqref{eq:H_W} when $W$ is not determined by a closed 1-form.

As shown in \cite{Cannarsa_Cheng_Zhang2014}, the local propagation of singularities along a Lipschitz curve was studied for viscosity solutions and Mather's barrier functions. In \cite{Cannarsa_Cheng2}, the relations between the critical points of the barrier functions and the homoclinic orbits with respect to Aubry sets was studied. The main methods used in \cite{Cannarsa_Cheng2} is the combination of Lasry-Lions regularization with standard kernel $|x-y|^2/2t$ and the critical point theory of mountain pass type.

\medskip

\noindent\textbf{B4.} What is the nature of the existence or non-existence of the critical points?

\medskip

\noindent\textbf{B5.} Is there a curvature condition characterizing the non-existence of critical point even for the classical mechanical systems, like what used in Cheeger-Gromoll's splitting theorem in Riemannian geometry (\cite{Cheeger_Gromoll1971})?

\medskip

\noindent\textbf{B6.} Let $u$ be a weak KAM solution with respect to a Tonelli Hamiltonian $H$. Invoking problem \textbf{A6}, at a critical point with respect to $u$, how should we change $H$ to understand the further propagation of singularities of $u$? Is this a way to solve problem \textbf{A7}?

\medskip

\noindent\textbf{B7.} Can we get more information, by using the intrinsic kernel in the process of Lasry-Lions regularization for the Mather's barrier function, to obtain the dynamical results from the critical points of the barrier functions?

\medskip

Recalling the results in \cite{Cannarsa_Peirone2001}, for the distance function $d_F$ with respect to a closed subset $F\subset\R^n$, some amazing results on the asymptotic properties of the unbounded component of $\text{Sing}\,(d_F)$ were obtained. We finish the list with the following problem:

\medskip

\noindent\textbf{B8.} What is the analogy of these results and how about the extensions for weak KAM solutions?

%\section*{Acknowledgements}
%Piermarco Cannarsa is partly supported by the University of Rome Tor Vergata (Consolidate the Foundation 2015) and Istituto Nazionate di Alta Matematica (GNAMPA 2017 Research Projects). Wei Cheng is partly supported by Natural Scientific Foundation of China (Grant No. 11631006 and No.11790272). Kaizhi Wang is partly supported by National Natural Science Foundation of China (Grant No. 11771283).

%\nocite{Castelpietra-Rifford}
%\nocite{Cannarsa-Sinestrari1995}
%\nocite{Clarke}
%\nocite{Fleming-Soner}
%\nocite{Vinter}
%\nocite{Lions}
%\nocite{Contreras}
%\nocite{Li-Nirenberg}

\bibliographystyle{abbrv}
\bibliography{mybib}

\end{document}